\pdfoutput=1
\RequirePackage{ifpdf}
\ifpdf 
\documentclass[pdftex]{sigma}
\else
\documentclass{sigma}
\fi

\numberwithin{equation}{section}

\newtheorem{thm}{Theorem}[section]
\newtheorem{lem}[thm]{Lemma}
\newtheorem{prop}[thm]{Proposition}
 { \theoremstyle{definition}
\newtheorem{Remark}[thm]{Remark} }

\usepackage{delarray}

\newcommand{\bbR}{\mathbb R}
\newcommand{\bbT}{\mathbb T}
\newcommand{\bbC}{\mathbb C}

\newcommand{\bbP}{\mathbb P}

\newcommand{\bbZ}{\mathbb Z}

\begin{document}

\allowdisplaybreaks

\newcommand{\arXivNumber}{1204.5701}

\renewcommand{\PaperNumber}{093}

\FirstPageHeading

\ShortArticleName{Orbital Linearization of Smooth Completely Integrable Vector Fields}

\ArticleName{Orbital Linearization of Smooth Completely\\ Integrable Vector Fields}

\Author{Nguyen Tien ZUNG~$^{\dag\ddag}$}

\AuthorNameForHeading{N.T.~Zung}

\Address{$^\dag$~School of Mathematics, Shanghai Jiao Tong University, \\
\hphantom{$^\dag$}~800 Dongchuan Road, Minhang District, Shanghai 200240, P.R.~China}

\Address{$^\ddag$~Institut de Math\'ematiques de Toulouse, UMR5219 CNRS, Universit\'e Paul Sabatier,\\
\hphantom{$^\dag$}~118 route de Narbonne, 31062 Toulouse, France}
\EmailD{\href{mailto:tienzung.nguyen@math.univ-toulouse.fr}{tienzung.nguyen@math.univ-toulouse.fr}}

\ArticleDates{Received July 04, 2017, in f\/inal form November 30, 2017; Published online December 12, 2017}

\Abstract{The main purpose of this paper is to prove the smooth local orbital linearization theorem for smooth vector f\/ields which admit a complete set of f\/irst integrals near a nondegenerate singular point. The main tools used in the proof of this theorem are the formal orbital linearization theorem for formal integrable vector f\/ields, the blowing-up method, and the Sternberg--Chen isomorphism theorem for formally-equivalent smooth hyperbolic vector f\/ields.}


\Classification{37G05; 58K50; 37J35}

\section{Introduction}

The main purpose of this paper is to show the following \emph{orbital linearization theorem} for smooth~($C^\infty$)
vector f\/ields which admit a complete set of f\/irst integrals near a nondegenerate singular point:

\begin{thm} \label{thm:SmoothLinearization}
Let $X$ be a smooth vector field in a neighborhood of $O = (0,\dots,0)$ in $\bbR^n$,
which vanishes at $O$ and satisfies the following conditions:
\begin{enumerate}\itemsep=0pt
\item[$i)$] $($complete integrability$)$: $X$ admits $n-1$ functionally independent smooth first integrals $F_1,\dots, F_{n-1}$,
i.e., $X(F_1) = \dots = X(F_{n-1}) = 0$ and ${\rm d}F_1 \wedge \dots \wedge {\rm d}F_{n-1} \neq 0$ almost everywhere;
\item[$ii)$] $($nondegeneracy~$1)$: the semisimple part of the linear part of $X$ at $O$ is non-zero, and the $\infty$-jets
of $F_1,\dots,F_{n-1}$ at $O$ are functionally independent $($i.e.,
the $\infty$-jet of ${\rm d}F_1 \wedge \dots \wedge {\rm d}F_{n-1}$ at $O$ is non-zero$)$;
\item[$iii)$] $($nondegeneracy~$2)$: If moreover $0$ is an eigenvalue of $X$ at $O$ with multiplicity $k \geq 1$,
then the differentials of the functions $F_1,\dots, F_k$ are linearly independent at $O$:
 ${\rm d}F_1(O) \wedge \dots \wedge {\rm d}F_k(O) \neq 0$. Then there exists a local smooth coordinate system $(x_1,\dots,x_n)$ in which $X$ can be written as
\begin{gather*}
X = F X^{(1)},
\end{gather*}
where $X^{(1)}$ is a semisimple linear vector field in $(x_1,\dots,x_n)$, and $F$ is a smooth first integral of $X^{(1)}$, i.e., $X^{(1)}(F) = 0$, with $F(O)=1$.
\end{enumerate}
\end{thm}

The above theorem is in fact more than mere orbital linearization: not only that $X$ is orbitally equivalent to its linear part
$X^{(1)}$, but also the factor $F$ in the expression $X = FX^{(1)}$ in a~normalized coordinate system is a
f\/irst integral of~$X$ and~$X^{(1)}$. In~\cite{Zung-Nondegenerate2015}, this kind of linearization is called
\textit{geometric linearization}.

The formal and analytic case of the above theorem also holds and was shown in \cite{Zung-Nondegenerate2015}
in a~more general context of integrable non-Hamiltonian systems of type $(p,q)$, i.e., with $p$ commuting
vector f\/ields and $q$ common f\/irst integrals, where $p+q = n$ is the dimension of the manifold. The vector f\/ields
that we study in this paper are integrable of type $(1,n-1)$, i.e., just one vector f\/ield and $n-1$ f\/irst integrals.

The nondegeneracy condition in Theorem \ref{thm:SmoothLinearization} is a bit stronger than the nondegeneracy
condition in \cite{Zung-Nondegenerate2015}: in \cite{Zung-Nondegenerate2015} the (formal or analytic)
vector f\/ield $X$ is called integrable \emph{nondegenerate} if it satisf\/ies the above conditions i) and ii), without the
need of condition iii). (A priori, the condition that ${\rm d}F_1 \wedge \dots \wedge {\rm d}F_{n-1}(O) \neq 0$ is quite stronger than the condition that the $\infty$-jet of ${\rm d}F_1 \wedge \dots \wedge {\rm d}F_{n-1}$ at~$O$ is not zero.) However, in fact, in the formal and analytic case, condition~iii) is a~simple consequence of the f\/irst two conditions and the theorem about the existence of (formal or analytic) Poincar\'e--Dulac normalization
\cite{Zung-Poincare2002,Zung-Nondegenerate2015}. On the other hand, in the smooth case,
we don't have a proof of the fact that condition~iii) follows
from conditions~i) and~ii) in general, though we do have a proof of this fact for dimension~2.

The rest of this paper is organized as follows. Section~\ref{section2} is devoted to some preliminary
results, including the classif\/ication
of nondegenerate singularities of completely integrable vector f\/ields into (strong/weak)
elliptic and hyperbolic cases (Lemma~\ref{lem:HypOrEl}), and
the normalization up to a f\/lat term (Proposition \ref{prop:LinearizationUpToFlat}). These preliminary results
are used in the proof of Theo\-rem~\ref{thm:SmoothLinearization} which is presented in Section~\ref{section3}.
Finally, in Section~\ref{section4}, we show that, at least in the case $n=2$, condition~iii) of
Theorem~\ref{thm:SmoothLinearization} is a consequence of the f\/irst two conditions, and can be dropped
from the formulation of the theorem (Theorem~\ref{thm:2Dsmoothlinearization}). We conjecture that
condition iii) is redundant in the higher-dimensional case as well.

This paper is part of our program of systematic study of the geometry and topology
of integrable non-Hamiltonian systems. In particular, Theorem~\ref{thm:2Dsmoothlinearization},
which is a ref\/inement of Theo\-rem~\ref{thm:SmoothLinearization} in the case of dimension~2,
is the starting point of our joint work with Nguyen Van Minh on the
local and global smooth invariants of integrable dynamical
systems on 2-dimensional surfaces~\cite{ZungMinh_2D2012}. In connection
with our results, we would
like to mention the theorem of Cha\-pe\-ron on smooth equivalence of formally equivalent weakly
hyperbolic systems \cite{Chaperon-WeakHyperbolic2013}, and its recent application to smooth
geometric linearization of some classes of integrable non-Hamiltonian
systems by Jiang~\cite{Jiang-WeakHyp2016}. We believe that Chaperon's techniques will be
a key element in our smooth linearization problem, see also~\cite{Zung-Integrable2016}.

\section{Preliminary results}\label{section2}

\subsection{Adapted f\/irst integrals}\label{subsection:adapted}

We have the following simple lemma, which is similar to
the well-known Ziglin's lemma \cite{Ziglin-Branching1982}.

\begin{lem}
 Let $G_1,\dots,G_m$ be $m$ formal series in $n$ variables which are functionally independent. Then
there exist $m$ polynomial functions of $m$ variables $P_1,\dots, P_m$ such that the homogeneous
$($i.e., lowest degree$)$ parts of the formal series of $P_1(G_1,\dots,G_m)$, $\dots$, $P_m(G_1,\dots,G_m)$ are functionally independent.
\end{lem}

The proof of the above lemma follows exactly the same lines as the proof of Ziglin of his lemma in
\cite{Ziglin-Branching1982}, and our situation is simpler than the situation of meromorphic
functions considered by Ziglin.

Let $X$ be a smooth completely integrable vector f\/ield with a singularity at $O$.
We will say that the smooth f\/irst integrals $F_1, \dots, F_{n-1}$ of $X$
are \emph{adapted} f\/irst integrals if
\begin{gather*}
 {\rm d}H_1 \wedge \dots \wedge {\rm d}H_{n-1} \neq 0 \quad \text{a.e.},
\end{gather*}
where $H_i = F_i^{(h_i)}$ denotes the homogeneous part (consisting of non-constant terms of lowest degree in the Taylor
expansion) of $F_i$ at $O$.
Using the above lemma to replace the f\/irst integrals $F_1, \dots, F_{n-1}$ of $X$ by appropriate
polynomial functions of them if necessary, from now on we can assume that $F_1,\dots, F_{n-1}$ are adapted.

\subsection[The eigenvalues of $X$]{The eigenvalues of $\boldsymbol{X}$}

The fact that $X$ admits $n-1$ f\/irst integrals implies that $X$ is very resonant at $O$. More precisely, we have:

\begin{lem} \label{lem:HypOrEl}
 Let $(X,F_1,\dots,F_{n-1})$ be smooth nondegenerate at $O$, i.e., they satisfy the conditions of Theorem~{\rm \ref{thm:SmoothLinearization}}.
Then the linear part of $X$ at $O$ is semisimple, and there is a positive number $\lambda > 0$ such that either all the eigenvalues of $X$ at $O$ belong to $\lambda \bbZ$, or all of them belong to $\sqrt{-1} \lambda \bbZ$.
\end{lem}

\begin{proof} We can assume that $H_1,\dots,H_{n-1}$ are functionally independent, where $H_i$ denotes the homogeneous part
of $F_i$. The equality $X(F_i) = 0$ implies that
\begin{gather*}
 X^{ss}(H_i) = X^{(1)}(H_i) = 0 \qquad \forall \, i=1,\dots,n-1,
\end{gather*}
where $X^{(1)}$ is the linear part of $X$, and $X^{ss}$ is the semisimple part of $X^{(1)}$ in the Jordan decomposition
(see, e.g.,~\cite{Zung-Poincare2002,Zung-Nondegenerate2015}). We can write
\begin{gather*}
 X^{ss} = \sum_{i=1}^n \lambda_i z_i \frac{\partial}{\partial z_i}
\end{gather*}
in a complex coordinate system. Recall that the ring of polynomial f\/irst integrals of $\sum\limits_{i=1}^n \lambda_i z_i \frac{\partial}{\partial z_i}$ is generated by the monomial functions $\prod\limits_{i=1}^n z_i^{a_i}$, which satisfy the resonance relation
\begin{gather} \label{eqn:Resonance}
 \sum_{i=1}^n a_i \lambda_i = 0.
\end{gather}
The fact that $H_1,\dots, H_{n-1}$ are independent implies that equation \eqref{eqn:Resonance} has $n-1$ linearly
independent solutions which belong to $\bbZ^n_+$, which in turn implies that there is a complex number~$\lambda$ such that $\lambda_1,\dots, \lambda_n \in \lambda \bbZ$. Remark that if the spectrum of $X^{ss}$ contains
a complex eigenvalue $\lambda_1 \in \bbC \setminus (\bbR \cup \sqrt{-1}\bbR)$, then its complex conjugate $\overline \lambda_1$
is also in the spectrum because $X$ is real, and $\lambda_1$ and $\overline \lambda_1$ cannot belong to $\lambda\bbZ$
at the same time for any $\lambda$. Thus any eigenvalue of $X^{ss}$ is either real or pure imaginary. If there is one real
non-zero eigenvalue, then we can choose $\lambda \in \bbR_+$, otherwise we can choose $\lambda \in \sqrt{-1}\bbR_+$.
Notice that $\lambda \neq 0$ because at least one eigenvalue of $X^{ss}$ is non-zero by our assumptions.

The common level sets of $H_1, \dots, H_{n-1}$ are 1-dimensional almost everywhere, and since both~$X^{(1)}$
and~$X^{ss}$ are tangent to these common level sets, we have that $X^{(1)} \wedge X^{ss} = 0$, which implies that
$X^{(1)}$ is semisimple, i.e., $X^{(1)} = X^{ss}$.
\end{proof}

With the above lemma, we can divide the problem into 4 cases (here $\bbR^* = \bbR \setminus \{0\}$):
\begin{enumerate}\itemsep=0pt
\item[I.] Strongly hyperbolic (or hyperbolic without eigenvalue 0): $\lambda_i \in \lambda\bbR^*$ $\forall\, i$.
\item[II.] Weakly hyperbolic (or hyperbolic with eigenvalue 0): $\lambda_i \in \lambda\bbR^* $ $\forall\, i > k \geq 1$,
$\lambda_1 = \dots = \lambda_k = 0$.
\item[III.] Strongly elliptic (or elliptic without eigenvalue 0): $\lambda_i \in \sqrt{-1}\lambda\bbR^*$ $\forall\, i$.
\item[IV.] Weakly elliptic (or elliptic with eigenvalue 0): $\lambda_i \in \sqrt{-1}\lambda\bbR^*$ $\forall\, i > k \geq 1$,
$\lambda_1 = \dots = \lambda_k$ $= 0$.
\end{enumerate}

\subsection{Linearization up to a f\/lat term} Using the geometric linearization theorem of \cite{Zung-Nondegenerate2015}
in the formal case, we get the following proposition:

\begin{prop}[linearization up to a f\/lat term] \label{prop:LinearizationUpToFlat}
 Assume that $X$ satisfies the hypotheses of Theorem~{\rm \ref{thm:SmoothLinearization}}. Then there is a local smooth
coordinate system $(x_1,\dots,x_n)$ in which $X$ can be written as
\begin{gather*}
X = FX^{(1)} + \text{\rm f\/lat},
\end{gather*}
where $X^{(1)}$ is the linear part of $X$ in the coordinate system $(x_1,\dots,x_n)$,
$F$ is a smooth first integral of $X^{(1)}$, and $\text{\rm f\/lat}$ means a smooth term which is flat at $O$.
\end{prop}

\begin{proof}
Denote by $\hat{X}$ (resp.~$\hat F_i$) the $\infty$-jet of $X$ (resp. $F_i$) at $O$: $\hat{X}$ is a formal vector f\/ield (resp. function)
at $O$. If $(X,F_1,\dots,F_{n-1})$ is smooth nondegenerate at $O$, then
$(\hat{X},\hat{F}_1,\dots,\hat{F}_{n-1})$ is a nondegenerate formal integrable system of type $(1,n-1)$ at $p$.
According to the geometric linearization theorem of \cite{Zung-Nondegenerate2015},
this formal integrable system can be linearized geometrically, i.e., there is a formal coordinate system
$(\hat x_1,\dots, \hat x_n)$ in which we have
\begin{gather} \label{eqn:FormalLinearization}
\hat X = \hat F \hat X^{(1)},
\end{gather}
 where $\hat X^{(1)}$ is the linear part of $\hat X$ in the formal coordinate system
$(\hat x_1,\dots, \hat x_n)$, and $\hat F$ is a~formal f\/irst integral of $\hat X^{(1)}$.
By the classical Hilbert--Weyl theorem (see, e.g., Theorem~4.2 of Chapter~XII of \cite{GSS-Bifurcation1988})
applied to the torus action associated to $X^{(1)}$ (see \cite{Zung-Poincare2002,Zung-Nondegenerate2015} for this associated torus action), we can write
\begin{gather} \label{eqn:f}
 \hat F = \hat f \big(Q_1(\hat x_1,\dots, \hat x_n),\dots, Q_m(\hat x_1,\dots, \hat x_n)\big),
\end{gather}
where $\hat f$ is a formal series and
$Q_1(\hat x_1,\dots, \hat x_n),\dots, Q_m(\hat x_1,\dots, \hat x_n)$ are homogeneous polynomials
generating the ring of polynomial f\/irst integrals of $\hat X^{(1)}$.
Using Borel theorem, we get a~smooth coordinate system $(x_1,\dots,x_n)$ whose $\infty$-jet is
$(\hat x_1,\dots, \hat x_n)$, and a smooth function~$f$ of~$m$ variables whose $\infty$-jet is $\hat f$. Put
\begin{gather} \label{eqn:F}
F (x_1,\dots,x_n) = f\big(Q_1(x_1,\dots,x_n),\dots,Q_m(x_1,\dots,x_n)\big).
\end{gather}
Then equations \eqref{eqn:FormalLinearization}, \eqref{eqn:f} and~\eqref{eqn:F} imply that $X = FX^{(1)} + \text{f\/lat}$
in the smooth coordinate system $(x_1,\dots,x_n)$.
\end{proof}

\subsection{Reduction to the case without eigenvalue 0}\label{subsection:ReductionToParametrizedCase}

Assume that $X$ has zero eigenvalue at $O$ with multiplicity $k$, and ${\rm d}F_1 \wedge \dots \wedge {\rm d}F_k = 0$,
i.e., we can use $F_1,\dots,F_k$ as the f\/irst $k$ coordinates in our local coordinate systems. Since
the vector f\/ield $X$ preserves $x_1,\dots,x_k$, we can view it as a $k$-dimensional family of vector f\/ields
on $(n-k)$-dimensional spaces
\begin{gather*}
 U_{c_1,\dots,c_k} = \{F_1 = c_1, \dots, F_k = c_k\}
\end{gather*}
(for $c_1,\dots,c_k$ small enough). It follows from the
usual implicit function theorem and the nondegeneracy condition
that on each $U_{c_1,\dots,c_k}$ there is a unique
point $O_{c_1,\dots,c_k}$ such that $X(O_{c_1,\dots,c_k})$ $= 0$, and moreover the point $O_{c_1,\dots,c_k}$
depends smoothly on $c_1,\dots,c_k$, the eigenvalues of $X$ at $c_1,\dots,c_k$, are non-zero.
It also follows from the formal independence of $F_1,\dots,F_n$ at $O$, that the
functions $F_{k+1},\dots,F_n$ are formally independent at every point $O_{c_1,\dots,c_k}$ provided
that $c_1,\dots,c_k$ are suf\/f\/iciently small. In other words, we have a $k$-dimensional family of nondegenerate
singularities of smooth completely integrable $(n-k)$-dimensional vector f\/ields $X_{c_1,\dots,c_k}$.
In order to normalize $X$, it suf\/f\/ices to normalize $X_{c_1,\dots,c_k}$ in a way which depends smoothly
on the parameter.

\section{Proof of Theorem \ref{thm:SmoothLinearization}}\label{section3}

We will always assume that the vector f\/ield
$X$ satisf\/ies the hypotheses of Theorem \ref{thm:SmoothLinearization}.
The fact that the linear part of $X$ is semisimple is established by Lemma~\ref{lem:HypOrEl}.
In view of Section~\ref{subsection:ReductionToParametrizedCase}, it suf\/f\/ices to prove
Theorem~\ref{thm:SmoothLinearization} for the cases without zero eigenvalue, by a proof whose
parametrized version also works the same.

\subsection{The hyperbolic case} Assume that $X$ is hyperbolic without eigenvalue 0.
According to Proposition \ref{prop:LinearizationUpToFlat}, we can write $X = Y + \text{f\/lat}$,
where $Y = FX^{(1)}$ is a smooth hyperbolic integrable vector f\/ield in normal form. Since
$X$ and $Y$ are hyperbolic and coincide up to a f\/lat term, Sternberg--Chen theorem
\cite{Chen-Vector1963, Sternberg} says that $X$ is locally smoothly isomorphic to $Y$, i.e.,
there is a smooth coordinate system in which $X$ can be written as $X = F X^{(1)}$, where $F$
is a smooth f\/irst integral of~$X^{(1)}$. Theorem~\ref{thm:SmoothLinearization} is proved in the
hyperbolic case without eigenvalue 0. This is also a special case
of a result of Kai Jiang~\cite{Jiang-WeakHyp2016} on smooth linearization of
weakly hyperbolic integrable vector f\/ields.

\subsection{The elliptic case} In this subsection, we will assume that all the eigenvalues of $X$ at $O$
are non-zero pure imaginary. Using Proposition \ref{prop:LinearizationUpToFlat}, we can assume that
$X = FX^{(1)} + \text{f\/lat}$ in a local smooth coordinate system $(x_1,\dots,x_n)$, where $F$ is a smooth
function such that $F(O) =1$. Put $Y = X/F$. Then $Y$ has the same f\/irst integrals as $X$, and
\begin{gather*} 
 Y = X^{(1)} + \text{f\/lat}.
\end{gather*}

The fact that $X$ is of strong elliptic type implies immediately that the dimension $n$ is even, the eigenvalues of $X$
at $O$ are $\pm \sqrt{-1} a_1,\dots, \pm \sqrt{-1} a_{n/2}$ where $a_1, \dots, a_{n/2}$ are positive real numbers,
and we can choose the coordinates $(x_1,\dots,x_n)$ such that
\begin{gather} \label{eqn:LinearElliptic}
 X^{(1)} = \sum_{i=1}^{n/2} a_i \left(x_{2i-1} \frac{\partial}{\partial x_{2i}} - x_{2i} \frac{\partial}{\partial x_{2i-1}}\right).
\end{gather}
According to Lemma \ref{lem:HypOrEl}, we can choose $\lambda > 0$ such that $a_1/\lambda, \dots, a_{n/2}/ \lambda$
are natural numbers whose greatest common divisor is 1.

\begin{lem} \label{lem:AllPeriodic}
Locally near $O$ all the orbits of $Y = X/F$ $($except the fixed point $O)$ are periodic, with periods which
are uniformly bounded above and below.
\end{lem}

\begin{proof}
 The vector f\/ield $({\rm d}F_1 \wedge \dots \wedge {\rm d}F_{n-1})
\lrcorner \big(\frac{\partial}{\partial x_1} \wedge \dots \wedge \frac{\partial}{\partial x_n}\big)$
is tangent to $Y$, and therefore it is divisible by $Y$ (by de~Rham division theorem,
because of the nondegeneracy condition, see, e.g., Appendix~A.2 of~\cite{DufourZung-Poisson2005}; in that appendix the de~Rham theorem is given
for dif\/ferential forms, but it works the same for vector f\/ields, via an isomorphism between the tangent bundle and the contangent bundle over
a manifold), i.e., we can write
\begin{gather} \label{eqn:3-3}
 ({\rm d}F_1 \wedge \dots \wedge {\rm d}F_{n-1})
\lrcorner \left(\frac{\partial}{\partial x_1} \wedge \dots \wedge \frac{\partial}{\partial x_n}\right) = GY,
\end{gather}
where $G$ is a smoth non-f\/lat function at $O$. Notice that the singular locus of the map $(F_1,\dots,$ $F_{n-1})\colon U \to \bbR^{n-1}$,
where $U \ni O$ is a small neighborhood of $O$ in $\bbR^n$, coincides with the zero locus of~$G$.

It is clear that, by continuity, the set of all points $x \in U$ such that the orbit of $Y$ through $x$ is periodic of period $\leq 3\pi/\lambda$
is a closed subset of $U$. We want to show that this set is actually equal to $U$ (provided that $U$ is small enough).
Consider the singular locus
\begin{gather*}
 S = \{x \in U \, | \,G(x) = 0 \} = \{x \in U \, | \,{\rm d}F_1 \wedge \dots \wedge {\rm d}F_{n-1} (x) = 0 \}.
\end{gather*}
Since $G$ is non-f\/lat at $O$, we can choose a coordinate system $(z_1,\dots, z_n)$ which is a linear transformation of the
coordinate system $(x_1,\dots,x_n)$, such that the homogeneous part $G^{(h)}$ of~$G$ has the form
\begin{gather} \label{eqn:3-5}
 G^{(h)} = z_1^h + \cdots
\end{gather}
(where $h$ denotes the degree of $G^{(h)}$), which implies that $\frac{\partial^h G}{\partial z_1^h} \neq 0$ in $U$.

Because $\frac{\partial^h G}{\partial z_1^h}$ does not vanish in $U$, by the classical Rolle's theorem on each line $\{z_2={\rm const}, \dots, z_n = {\rm const}\}$ in $U$ there are at most $h$ zeros of the function $G$, the intersection of the singular locus $S$ with each line $\{z_2={\rm const}, \dots, z_n = {\rm const}\}$ in $U$ consists of at most $n-1$ points, and the function $G$ is not f\/lat at any point of $S$.

Due to the ellipticity of the vector f\/ield $X^{(1)}$, there must be at least
one index $j \neq 1$ such that
the coef\/f\/icient of the monomial term $z_1 \frac{\partial}{\partial z_j}$ in $X^{(1)}$ in the coordinate system $(z_1,\dots, z_n)$ is not zero.
Without loss of generality, we may assume that the coef\/f\/icient of
the monomial term $z_1 \frac{\partial}{\partial z_n}$ in $X^{(1)}$
is not zero.

Consider the local hyperplane
\begin{gather*} P = \{z_n = 0\} \subset U\end{gather*}
with the coordinate system $(z_1,\dots, z_{n-1})$.
For each $\epsilon > 0$ suf\/f\/iciently small, denote by
\begin{gather*} q_{\epsilon} = (z_1= \epsilon, z_2=0,\dots,z_n=0) \in P\end{gather*}
the point in $U$ whose coordinate $z_1$ is equal to $\epsilon$
and the other coordinates $z_i$ vanish for every $i \geq 2$.
Then the vector f\/ields $X^{(1)}$ and $Y$ are transversal to $P$
at every point $q_{\epsilon}$ such that $\epsilon > 0$ is suf\/f\/iciently
small.

Recall from Section \ref{subsection:adapted} that we can, and will,
assume that $F_1,\dots, F_n$ to be adapted f\/irst integrals, i.e.,
the homogeneous part of ${\rm d}F_1 \wedge \dots \wedge {\rm d} F_n$ is equal to
${\rm d}H_1 \wedge \dots \wedge {\rm d} H_n$, where $H_i = F_i^{(h_i)}$ is the homogeneous part
of $F_i$ for each $i=1,\dots, n-1$ and $\deg H_i = h_i$.
In particular, we have
$ h + 1 = \sum\limits_{i=1}^{n-1} (h_i - 1)$,
where $h$ is the degree of the homogeneous part~$G^{(h)}$
of the function $G$ in formulas~\eqref{eqn:3-3} and~\eqref{eqn:3-5},
and $h+2 \geq h_i$ for any $i=1,\dots, n-1$.

Denote by $B^{n-1}(q_\epsilon, \epsilon^{h+2})$ the
$(n-1)$-dimensional ball of center $q_\epsilon$ and radius $\epsilon^{h+2}$
on the local hyperplane $ P = \{z_n = 0\}$. We observe that, for every
$\epsilon $ suf\/f\/iciently small, the restriction of the map
$(F_1,\dots, F_{n-1})$ to $B^{n-1}(q_\epsilon, \epsilon^{h+2})$
is an injective map
from $B^{n-1}(q_\epsilon, \epsilon^{h+2})$ to $\mathbb{R}^{n-1}$.

Indeed, consider any two distinct points ${\bf p},{\bf q} \in B^{n-1}(q_\epsilon, \epsilon^{h+2})$, ${\bf p} \neq {\bf q}$.
Consider the constant unit vector f\/ield $Z_{{\bf p},{\bf q}} = \frac{{\bf q} - {\bf p}}{\|{\bf q} - {\bf p} \|}$
on $U$ with respect to the coordinate system $(z_1,\dots, z_n)$ which maps
$U$ onto a neighborhood of the origin in the Euclidean space $\mathbb{R}^n$.
The dif\/ference ${\bf q} - {\bf p}$ and the norm $\|{\bf q} - {\bf p} \|$
(i.e., the distance from ${\bf p}$ to ${\bf q}$) are taken with respect
to this Euclidean structure. Due to the fact that our
f\/irst integrals are adapted, at least one of the functions
$Z_{{\bf p},{\bf q}}(F_1), \dots, Z_{{\bf p},{\bf q}}(F_{n-1})$,
say $Z_{{\bf p},{\bf q}}(F_k)$, has a homogeneous part of the type
\begin{gather*} c_{{\bf p},{\bf q}} z_1^{h_k -1} + \cdots,
\end{gather*}
with a coef\/f\/icient $c_{{\bf p},{\bf q}} \neq 0$, and $k$ can be chosen
as a function of ${\bf p}$ and ${\bf q}$ such that $|c_{({\bf p},{\bf q})}|$
is uniformly bounded by positive constants. Using standard division techniques
(see, e.g., the Malgrange's preparation
theorem \cite{Malgrange-Ideal1967}, but the situation here is much simpler than this general theorem), one can decompose $Z_{{\bf p},{\bf q}}(F_k)$ as
\begin{gather*} Z_{{\bf p},{\bf q}}(F_k) = f_{1, {\bf p},{\bf q}}z_1^{h_k-1}
+ \sum_{i=2}^n z_i f_{i, {\bf p},{\bf q}},\end{gather*}
where $f_{i, {\bf p},{\bf q}}$ are smooth functions on $U$ depending uniformly
continuously on the unit vector~$Z_{{\bf p},{\bf q}}$ (as long as the index $k$ remains the same), and with $f_{1, {\bf p},{\bf q}}(0) = c_{{\bf p},{\bf q}}$.
For the points on $B^{n-1}(q_\epsilon, \epsilon^{h+2})$ we have that
$\epsilon - \epsilon^{h+2} \leq z_1 \leq \epsilon + \epsilon^{h+2}$
while $|z_i| \leq \epsilon^{h+2} \leq \epsilon^{h_k}$ for all $i \geq 2$,
hence in the above expression of $Z_{{\bf p},{\bf q}}(F_k)$ the terms
$z_i f_{i, {\bf p},{\bf q}}$ with $i \geq 2$ are very small compared to the
term $f_{1, {\bf p},{\bf q}}z_1^{h_k-1}$, and so in particular the sign of
$Z_{{\bf p},{\bf q}}(F_k)$ does not change on $B^{n-1}(q_\epsilon, \epsilon^{h+2})$ for~$\epsilon$ suf\/f\/iciently small, which implies that
$F_k({\bf q}) - F_k({\bf p}) \neq 0$ by the mean value theorem. Thus, we have shown the injectivity of $(F_1,\dots, F_{n-1})$ on $B^{n-1}(q_\epsilon, \epsilon^{h+2})$ for every $\epsilon$ suf\/f\/iciently small.

Consider now the Poincar\'e map (i.e., the f\/irst return map), denoted by $\phi$,
def\/ined on the ball $B^{n-1}(q_\epsilon, \epsilon^{h+2}) \subset P$
of the f\/low of $Y$ in $U$
(a priori the image of this map may lie a bit outside of
$B^{n-1}(q_\epsilon, \epsilon^{h+2})$ but on the same
local hyperplane $P$). A priori this map does not necessarily f\/ix
the point $q_\epsilon$. But due to the fact the f\/low of $X^{(1)}$ is
periodic (in particular, the Poincar\'e map for $X^{(1)}$ is the identity map)
and the fact that $Y = X^{(1)} + \text{f\/lat}$ (which implies that the
Poincar\'e map $\phi$ for $Y$ deviates from the Poincar\'e map for $X$ by a
f\/lat term), we have that the distance from $q_\epsilon$ to
$\phi(q_\epsilon)$ is a f\/lat function in $\epsilon$. In particular,
for every $\epsilon$ suf\/f\/iciently small we have
\begin{gather*}
{\rm d}(q_\epsilon , \phi(q_\epsilon)) < \epsilon^{h+2},\end{gather*}
where $d$ denotes the Euclidean distance in the coordinate system $(z_1,\dots, z_n)$, and hence the point
$\phi(q_\epsilon)$ lies in the
$(n-1)$-dimensional ball $B^{n-1}(q_\epsilon, \epsilon^{h+2})$.

Due to the invariance of the functions $F_i$ with respect to
the vector f\/ield $Y$,
and hence with respect to the Poincar\'e map $\phi$, we also have that
the points $q_\epsilon$ and $\phi(q_\epsilon)$ have the same image
under the map $(F_1,\dots,F_{n-1})$. But this map is injective on the ball $B^{n-1}(q_\epsilon, \epsilon^{h+2})$
which contains these two points, so in fact these two points must coincide,
i.e., we have $q_\epsilon = \phi(q_\epsilon)$,
and the orbit of the f\/low of $Y$ through the point $q_\epsilon$ is a periodic orbit, and the period of this orbit
is equal to $2\pi/\lambda$ plus a small error term which tends to~0 faster
than any power of $\epsilon$ when~$\epsilon$ tends to~0.

Denote by $V$ the path-connected component of $U \setminus S$ which contains the points $q_\epsilon$. ($V$ is not equal $U \setminus S$ in general). Then
the orbit of $Y$ through any point $q \in V$ is also periodic and its period is close to $2\pi/\lambda$
(the dif\/ference between the period and $2\pi/\lambda$ tends to 0 uniformly when the radius of $U$ tends to 0). This
fact can be proved easily by showing that the set of points of~$V$ which satisf\/ies the mentioned property is closed
and open in~$V$ at the same time: closed due to the continuity, and open because $(F_1,\dots,F_{n-1})$
is regular in~$V$ and is preserved by the f\/low of~$Y$.

Let $q \in S$ be a point in the locus $S$ which also lies on the boundary of $V$. Then by continuity, there is also
a number $T$ near $2\pi/\lambda$ such that the time-$T$ f\/low of $Y$ f\/ixes the point $q$. In other words, the orbit
of~$Y$ through~$q$ is also periodic, and the period is equal to~$T$ or a fraction $T/m$ of $T$ for some natural
number~$m$. As before, consider a $(n-1)$-dimensional
ball $B^{n-1}(q, \delta)$ which is centered at $q$ and orthogonal to $Y(q)$, for some $\delta > 0$ small enough.
Consider the Poincar\'e map $\phi$ of $Y$ on $B^{n-1}(q, \delta)$ corresponding to the time $T$ (i.e., if the period of the
orbit through~$q$ is $T/m$ then consider the $m$-time iteration of the usual Poincar\'e map). Since the intersection of
${B^{n-1}(q, \delta)}$ with $V$ contains an open subset of $B^{n-1}(q, \delta)$ whose closure contains~$q$, and
the Poincar\'e map is identity on that open subset by the above considerations, the Poincar\'e map on $B^{n-1}(q, \delta)$
is equal to the identity map plus a f\/lat term at $q$. On the other hand, this Poincar\'e map must preserve the
map $(F_1,\dots,F_{n-1})|_{B^{n-1}(q, \delta)}$, and the determinant of the dif\/ferential of this map is not f\/lat at $q$.
It implies that the Poincar\'e map must be identity in a small neighborhood of $q$ in ${B^{n-1}(q, \delta)}$. Thus, we can ``engulf'' the set of points shown to have periodic orbits from $V$ to a larger open subset of $U$ which contains the boundary of $V$. Continuing this
engulf\/ing process, we get that the set of points in $U$ having periodic orbits is actually the whole~$U$.
\end{proof}

\looseness=-1 We will linearize $Y = X/F$ orbitally, and then deduce the normalization of $X$ from this linearization. In order to do that,
let us consider the blow-up of $\bbR^n$ at $O$, which will be denoted by
\begin{gather*}
 p\colon \ E \to U,
\end{gather*}
where $U \ni O$ is a neighborhood of $O$ in $\bbR^n$ and $p^{-1}(O) \cong \bbR \bbP^{n-1}$ is the exceptional
divisor of the blow-up in $E$. We will need the following simple lemma, whose proof is straightforward:

\begin{lem} \label{lem:blowup}
With the above notations, a function $G$ or a vector field $Z$ is flat at $O$ in $U$ if and only if its pull-back to
$E$ via the projection map $p$ is flat along $p^{-1}(O)$ in $E$.
\end{lem}

(In the above lemma, $G$ and $Z$ are arbitrary, and the pull back from
$U$ to $E$ means the con\-ti\-nuous extention from $E \setminus p^{-1}(O)$
to $E$ of the pull-back by the dif\/feomorphism $p\colon E \setminus p^{-1}(O)$ $\to U \setminus \{O\}$, if such a continuous extension exists.)

Denote by $\tilde G$ (resp.~$\tilde Z$) the pull-back of a function $G$ (resp.\ vector f\/ield $Z$) via the projection
map $\pi\colon E \to U$ of the blow-up. Then we have
\begin{gather*}
\tilde Y = \tilde X^{(1)} + \tilde Z
\end{gather*}
in $E$, where $\tilde Z$ is vector f\/ield which is f\/lat along $p^{-1}(O)$, and $\tilde X^{(1)}$ is a smooth periodic vector f\/ield
in $E$ of period $2\pi/\lambda$. By Lemma~\ref{lem:AllPeriodic}, the orbits of $\tilde Y$ are closed, with periods close
to the period of $\tilde X^{(1)}$. Due to the f\/latness of $Z$ along $p^{-1}(O)$, the period of $\tilde Y$ at the points in $E$
is equal to $2\pi/\lambda$ plus a smooth function on $E$ which is f\/lat along $p^{-1}(O)$. Projecting $\tilde Y$ back to~$U$
and using Lemma~\ref{lem:blowup}, we get a smooth period function $P = 2\pi/\lambda + \text{f\/lat}$
(which is invariant on the orbits)
such that $PY$ is periodic of period~1.
In other words, $PY$ generates a smooth $\bbT^1$-action. Using the classical Cartan--Bochner
smooth linearization theorem
for compact group actions, we f\/ind a smooth coordinate system, which we will denote again by $(x_1,\dots,x_n)$,
in which $PX/F = PY$ is a linear vector f\/ield, i.e., in which we have
\begin{gather} \label{eqn:OrbitalLinearizationElliptic}
X = GX^{(1)},
\end{gather}
where $G$ is a smooth function and $X^{(1)}$ is a linear vector f\/ield which satisf\/ies formula \eqref{eqn:LinearElliptic}.

A priori, the function $F$ given by Proposition~\ref{prop:LinearizationUpToFlat} is not a f\/irst integral of $X$
(though it is a f\/irst integral of the linear part of $X$ in some coordinate system), and so the function
$G = 2\pi F/P \lambda$ in formula~\eqref{eqn:OrbitalLinearizationElliptic} is not a f\/irst integral of $X$ either.
But we can normalize further in order to change~$G$ into a f\/irst integral. Indeed, by the arguments presented above,
we can assume that~$G$ is a smooth f\/irst integral of $X$ plus a~f\/lat term, or we can write $G = G_1 (1+ \text{f\/lat})$, where
$G_1$ is a~f\/irst integral of $X$. Normalizing the new vector f\/ield $Y = X/G_1$ instead of the old $Y = X/F$, we get a~new smooth coordinate system in which $PY = (2\pi/\lambda) X^{(1)}$, where $P$ is the period function of the new vector f\/ield $Y$,
and it is a smooth f\/irst integral of the type $\text{const} + \text{f\/lat}$. In this new coordinate system we have that $X$ is equal to
its linear part times a f\/irst integral, and Theorem~\ref{thm:SmoothLinearization}
is proved in the elliptic case, i.e., without eigenvalue ~0.

Since our proof for the strong hyperbolic case and the strong elliptic case also works for smooth families of integrable vector f\/ields,
Theorem~\ref{thm:SmoothLinearization} is proved.

\begin{Remark}
According to a theorem of Schwarz \cite{Schwarz-Invariant1975},
the smooth f\/irst integral $F$ in the normal form in the elliptic case can also be written as
\begin{gather*}
 F = f\big(Q_1(x_1,\dots,x_n), \dots, Q_m(x_1,\dots,x_n)\big),
\end{gather*}
where $Q_1(x_1,\dots,x_n), \dots, Q_m(x_1,\dots,x_n)$ are homogeneous polynomials which generate
the ring of polynomial f\/irst integrals of the linear vector f\/ield $X^{(1)}$.
\end{Remark}

\section{The case of dimension 2}\label{section4}

The aim of this section is to show that condition iii) in Theorem~\ref{thm:SmoothLinearization}
is redundant at least in the case of dimension~2. More precisely, we have:

\begin{thm} \label{thm:2Dsmoothlinearization}
Let $X$ be a smooth vector field in a neighborhood of $O = (0,0)$ in $\bbR^2$,
which vanishes at $O$ and satisfies the following conditions:
\begin{enumerate}\itemsep=0pt
\item[$i)$] $($complete integrability$)$: $X$ admits a smooth first integral $F_1$;
\item[$ii)$] $($nondegeneracy$)$: the semisimple part of the linear part of $X$ at $O$ is non-zero, and the $\infty$-jet
of $F_1$ at $O$ is non-constant.
Then there exists a local smooth coordinate system $(x,y)$ in which $X$ can be written as
\begin{gather*}
X = F X^{(1)},
\end{gather*}
where $X^{(1)}$ is a semisimple linear vector field in $(x,y)$, and $F$ is a smooth first
integral of~$X^{(1)}$.
\end{enumerate}
\end{thm}

\begin{proof}
Remark that, in the case of dimension 2, there are only 3 possibilities: elliptic without zero eigenvalue,
hyperbolic without zero eigenvalue, and hyperbolic with zero eigenvalue. The f\/irst two possibilities
are covered by Theorem \ref{thm:SmoothLinearization}. It remains to prove Theorem \ref{thm:2Dsmoothlinearization}
for the case when $X$ has one eigenvalue equal to~0. By Proposition \ref{prop:LinearizationUpToFlat}, we can assume
that
\begin{gather*}
X = F(y) x \frac{\partial}{\partial x} + \text{f\/lat}_1 \frac{\partial}{\partial x} + \text{f\/lat}_2 \frac{\partial}{\partial y}
\end{gather*}
in a smooth coordinate system $(x,y)$, where $\text{f\/lat}_1$ and $\text{f\/lat}_2$ are two f\/lat functions, and $F(0) \neq 0$.

 Denote by
\begin{gather*}
 S = \{ q \in U \, | \, X(q) = 0\}
\end{gather*}
the singular locus of $X$ near $O$, where $U$ denotes a small neighborhood of $O$ in $\bbR^2$. The main point is to prove
that $S$ is a smooth curve. If $S$ is a smooth curve, then we can write $S = \{x=0\}$, the vector f\/ield $X$ is divisible by $x$,
i.e., $Y = X/x$ is still a smooth vector f\/ield, which is non-zero at $O$, and therefore locally rectif\/iable and admits a
f\/irst integral $G$ such that ${\rm d}G(0) \neq 0$. But $G$ is also a f\/irst integral of $X$, so condition iii) of Theorem \ref{thm:SmoothLinearization}
is also satisf\/ied, and Theorem \ref{thm:2Dsmoothlinearization} is reduced to a particular case of Theorem \ref{thm:SmoothLinearization}.

Denote by
\begin{gather*}
 S_1 = \{(x,y) \in U \, | \, F(y)x + \text{f\/lat}_1 (x,y) = 0 \}
\end{gather*}
the set of points where the $\frac{\partial}{\partial x}$-component of $X$ vanishes. It is clear that $S \subset S_1$, and $S_1$ is a~smooth curve
tangent to the line $\{ x= 0 \}$ at $O$ by the inverse function theorem. We will show that $S = S_1$.

Consider the open cone
\begin{gather*}
 C= \{(x,y) \in U \, | \, |x| < |y| \}.
\end{gather*}
Clearly, $S_1 \subset C \cup \{O\}$ (provided that $U$ is small enough). The non-f\/lat f\/irst integral $F_1$ of $X$ in the coordinate system
$(x,y)$ has the type
\begin{gather*}
F_1 = f(y) + \text{f\/lat},
\end{gather*}
where $f(y) = a_hy^h + \text{h.o.t.}$ is a non-f\/lat smooth function. It implies that the level sets of $F_1$ in the cone $C$ are smooth curves (because $\frac{\partial F_1}{\partial y} \neq 0$ in this cone) which are nearly tangent to the lines $\{y= \text{const}\}$ (because $\big|\frac{\partial F_1}{\partial x}\big|$ is very small compared to $\big|\frac{\partial F_1}{\partial y}\big| $ in the cone). In particular, each level set of $F_1$ in $C$ intersects with $S_1$ at exactly 1 point. Since $X$ is tangent to these level sets, and the $\frac{\partial}{\partial x}$-component of $X$ vanishes at the intersection points of these level sets with~$S_1$, it follows that $X$ itself vanishes at these intersection points. But every point of $S_1$ is an intersection point of~$S_1$ with a level set of~$F_1$. Thus $X$ vanishes on $S_1$, and we have $S= S_1$.
\end{proof}

\begin{Remark}
Two-dimensional elliptic-like vector f\/ields, i.e., those vector f\/ields whose orbits near a singular point are closed,
are also called \emph{centers} in the literature. There is a recent
interesting theorem of Maksymenko~\cite{Maksymenko-Symmetries2010} about the orbital linearization of the center,
without the assumption on the existence of a f\/irst integral, but with an assumption on the periods of the periodic
orbits. Maksymenko's theorem is similar to and a bit stronger than the
elliptic case of Theorem \ref{thm:2Dsmoothlinearization} because
his assumptions are weaker, and the conclusions are the same. His proof is also based on the formal normalization
and the blowing-up method.
\end{Remark}

\begin{Remark}
Some of the arguments of the proof of Theorem \ref{thm:2Dsmoothlinearization} are still valid in the $n$-dimensional
case where 0 is an eigenvalue with multiplicity $k \geq 1$. In particular, one can still show that, even without condition iii)
of Theorem \ref{thm:SmoothLinearization}, the local singular locus of $X$ is still a~smooth $k$-dimensional manifold. However,
it is more dif\/f\/icult to show that there is still a local regular invariant $(n-k)$-dimensional foliation. If one can show the existence
of this regular invariant foliation, then one can drop condition iii) from the statement of
Theorem \ref{thm:SmoothLinearization} because it is a consequence of the f\/irst two conditions. Maybe it is possible to use
 the techniques of Belitskii--Kopanskii \cite{BK-Equivariant2002} together with a kind of desingularization of the f\/irst
integrals in order to show the existence of an invariant regular foliation, but we don't have a proof so far.
\end{Remark}

\begin{Remark}
As pointed out by a referee of this paper, there is a less elementary but more
dynamical proof of Theorem \ref{thm:2Dsmoothlinearization} which uses a
$C^r$-central manifold of $X$, where $r$ can be arbitrarily large.
\end{Remark}

\subsection*{Acknowledgement}

The f\/irst version of this manuscript was available since 2012 as an unpublished
preprint (see arXiv:1204.5701v1). It was then revised and submitted during
the author's stay at the School of
Mathematical Sciences, Shanghai Jiao Tong University, as a visiting
professor in 2017. He would like to thank Shanghai Jiao Tong University, and
especially Tudor Ratiu, Jianshu Li, and Jie Hu for the invitation,
hospitality and excellent working conditions.

The authors would also like to thank the referees of this paper for many
pertinent remarks which helped improve the presentation of the paper.

\pdfbookmark[1]{References}{ref}
\LastPageEnding

\end{document}